\documentclass[11pt,twoside]{amsart}
\usepackage{mathrsfs}
\usepackage{mathrsfs}
\usepackage{txfonts}
\usepackage{bbm}
\usepackage{mathrsfs}
\usepackage{amsfonts}
\usepackage{amsmath}
\usepackage{amssymb}
\usepackage{wasysym}
\usepackage{psfrag}
\usepackage{graphics,epsfig,amsmath}  
\usepackage{comment}
\usepackage{enumerate}
\newtheorem{Definition}{Definition}

\newtheorem{Lemma}{Lemma}
\newtheorem{Theorem}{Theorem}

\newtheorem{Remark}{Remark}

\title[Total destruction of invariant tori] {Total destruction of invariant tori for the generalized Frenkel-Kontorova model}

\author[X. Su]{Xifeng Su}
\address{Academy of Mathematics and Systems Science,Chinese Academy of Sciences,
No.55 Zhongguancun East Road, Beijing, 100190, CHINA}
\email{billy3492@gmail.com, xfsu@amss.ac.cn}

\author[L. Wang]{Lin Wang}
\address{Department of Mathematics, Nanjing University, 22 Hankou Road, Nanjing 210093,CHINA}
\email{linwang.math@gmail.com}

\begin{document}
\maketitle

\begin{abstract}
We consider generalized Frenkel-Kontorova models on higher
dimensional lattices. We show that the invariant tori which are
parameterized by continuous hull functions can be destroyed by small
perturbations in the $C^r$ topology with $r<1$.
\end{abstract}
\begin{quote}
\footnotesize {\it Key words.} invariant tori, variational methods,
hull functions, Percival Lagrangian
\end{quote}
\begin{quote}
\footnotesize {\it AMS subject classifications (2000).} 82B20,
37A60, 49J40, 58F30, 58F11 \end{quote} \vspace{20pt}




\section{Introduction}
The standard Frenkel-Kontorova model:
\begin{equation}\label{standard FK}
\mathscr{L}(u)=\sum_{i\in\mathbb{Z}}\frac{1}{2}(u_i-u_{i+1})^2+\frac{\lambda}{2\pi}
\cos(2\pi u_i),
\end{equation} is the most famous example of models to describe
rather crude microscopic theory of plasticity due to dislocations.
It also has been considered as a model of deposition of material
over a periodic 1-dimensional substratum. The $u_i$ denotes the
position of the $i$th particle and the terms
$\frac{1}{2}(u_i-u_{i+1})^2$ and $\frac{\lambda}{2\pi} \cos(2\pi
u_i)$ model the energy of interaction of nearest neighboring
particles and the energy of interaction with the substratum
respectively.

 In this paper, we
are concerned with the destruction of invariant tori for the
generalized Frenkel-Kontorova models on higher dimensional lattices,
i.e. the variational problem for ``configurations''
$u:\mathbb{Z}^d\rightarrow \mathbb{R}$:
\begin{equation}\label{variational problem}
\mathscr{L}(u)= \sum_{i\in\mathbb{Z}^d} \sum_{j=1}^{d}
H_j(u_i,u_{i+e_j})
\end{equation}where $H_j: \mathbb{R}^2\rightarrow \mathbb{R}$ and
$e_j\in\mathbb{R}^d$ is the unit vector in the $j$th direction. The
model \eqref{variational problem} is a natural generalization of
Frenkel-Kontorova models to more complicated physical lattices in
the point of view of the solid state motivation.

We say the Frenkel-Kontorova model admits an invariant torus with
rotation vector $\omega$ if the configuration $u$ with rotation
vector $\omega$ can be parameterized by some continuous hull
function (see \eqref{parameterization} below). The case $d=1$
corresponds to twist maps. Destruction of invariant tori for this
model is closely related to the instability problems in solids, such
as diffusions on lattices \cite{Sinisa'99} and so on.

The results in this direction, to the best of our knowledge, is the
work of V. Bangert who proved the existence of gaps after large
perturbations in the elliptic PDE case \cite{Bangert'87}. More
precisely, under appropriate non-degenerate conditions on integrand
$F$, \cite{Moser'86,Bangert'87a} proved the existence of Mather set
$\mathscr{M}_\alpha^{rec}(F)$. Here $\mathscr{M}_\alpha^{rec}(F)$ is
a natural generalization of the set
$\mathscr{M}_\alpha^{rec}(F_0)=\{v~|~v(x)=\alpha\cdot x+ v_0, v_0\in
\mathbb{R}\}$ for the Dirichlet integrand $F_0(x,v,v_x)=|v_x|^2$.
\cite{Bangert'87} provided some examples of integrands $F$ such that
the graphs of the functions $v\in \mathscr{M}_\alpha^{rec}(F)$ only
form a lamination, i.e. a foliation with gaps, for all rotation
vectors $\alpha\in \mathbb{R}^n$ with $|\alpha|$ smaller than some
arbitrary constant.

From a quite different point of view from V. Bangert's, we will
prove the existence of gaps for the generalized Frenkel-Kontorova
models after small perturbations by variational methods.

More precisely, we will prove the following theorem:

\begin{Theorem}\label{fundamental lemma}
There exists a sequence of $C^\infty$ functions
$\{H_j^n\}_{n\in\mathbb{N}}$ converging to the integrable system
$H_0$ with $H_0(x,x')=1/2(x-x')^2$ in the $C^{1-\epsilon}$ topology
such that for $n$ large enough, the generalized Frenkel-Kontorova
models produced by $H_j^n$ admit no invariant tori, where
$\epsilon>0$ is a small constant independent of $n$.
\end{Theorem}

\section{Comparison with the literatures}
The study of invariant objects on which the dynamics can be
conjugated to a rotation is one of the central subjects of dynamical
systems and mathematical physics. The celebrated KAM theory ensures
the  persistence of KAM tori after small perturbation for integrable
Hamiltonian systems with non-degeneracy conditions and sufficient
differentiability (see \cite{Rafael'01} and references therein).

The Aubry-Mather theory which had its origin in the work of S. Aubry
\cite{Aubry'83} and J. Mather \cite{Mather'82} resembles KAM thoery.
It produces invariant objects by variational methods which could be
Cantor sets in the case of twist maps
\cite{Percival'79,Percival'80}. For higher dimensional lattices, the
corresponding Aubry-Mather theory is recently established in
\cite{Blank'89,Rafael'97,Rafael'98,Rafael'07,SL}.

The study of non-existence of invariant tori is the complementary to
the KAM theory and Aubry-Mather theory. The critical borderline
between existence and non-existence of invariant tori in the higher
dimensional case is still open and seems to be rather complicated
since there are some new phenomena involved.

The problem of destruction of invariant tori under small
perturbations has two different flavors as follows:
\begin{enumerate}
  \item  [a)]destruction of invariant torus with a given rotation
  vector,
  \item  [b)]total destruction of invariant tori for all rotation
  vectors.
\end{enumerate}

Case a) is closely related to the arithmetic property of the given
rotation vector. Roughly speaking, there is a balance among the
arithmetic property of the rotation vector, the regularity of the
perturbation and its topology. We mention \cite{H2, H3, Mather'88,
Forni'94} for destruction of invariant circles for twist maps and
\cite{Cheng} for destruction of KAM torus for Hamiltonian system of
multi-degrees of freedom. In particular, \cite{H2} considered
non-existence of invariant circle with arbitrary given rotation
number. \cite{W} provided a variational proof of Herman's result.
\cite{Mather'88, Forni'94} are concerned with destruction of
invariant circles with certain rotation numbers which can be very
fast approximated by rational numbers.

Comparing with Case a), due to the absence of arithmetic property,
Case b) is much harder and less results are obtained. Naturally, it
results in the loss of the regularity of the perturbation to
overcome the absence of arithmetic property. A criterition of total
destruction plays a crucial role. For twist maps, total destruction
of invariant circles is equivalent to the existence of a unbounded
orbit. Some criteria for non-existence of invariant tori appeared in
\cite{Chirikov'79} and examples of destruction for $C^1$ small
perturbation happened in \cite{Takens'71}. Much stronger result was
obtained by Herman based on a different geometrical criterition. In
\cite{H2}, he proved that for all rotation numbers, invariant
circles can be destroyed by $C^\infty$ perturbations close to $0$
arbitrarily in $C^{3-\epsilon}$ topology for the generating
functions. In \cite{H4}, Herman extended the above result to the
symplectic twist map of multi-degrees of freedom.

\begin{Remark}
All these subjects have a very long story. For example Appendix A of
\cite{CdlL'09} contains a mini-survey of the constructive methods to
verify the destruction of invariant tori for twist maps.
\end{Remark}

\section{Preliminaries}
 In order to use
variational methods, it is important to set up the Percival
Lagrangian approach to the original variational problem
\eqref{variational problem} in the following way (see
\cite{Percival'79,Percival'80, SL} for more details).

We assume that the configurations $u$ are parameterized by a
function $h$--the hull function--and a frequency $\omega\in
\mathbb{R}^d$ such that
\begin{equation}\label{parameterization}
u_i=h(\omega\cdot i).
\end{equation} where $i\in \mathbb{Z}^d$ and $\cdot$ is the usual inner product in $\mathbb{R}^d$.

Heuristically, considering a big box and normalizing the Lagrangian
in that big box, when the size of that big box goes to infinity in
some sense, we are led to considering
\begin{equation}\label{Percivalian}
   \mathscr{P}_\omega(h) =
   \sum_{j=1}^d \int_0^{1} H_{j}(h(\theta), h(\theta +\omega_j))d\theta.
\end{equation}

This imprecise derivation shows that $\mathscr{P}_\omega(h)$ has a
direct physical interpretation as the average energy per volume.

Finding the ground states for the variation problem
\eqref{variational problem} is transformed into finding the
minimizers of the variational problem \eqref{Percivalian}.  The
rigorous study of all these is included in \cite{SL}.

We also note that this formalism can be used as the basis of KAM
theory to produce smooth solutions under some assumptions such as
Diophantine properties of the frequencies, the system is close to
integrable, etc. (see \cite{SZ'89, CdlL'09}).

Assume $H_j$ with the following form:
\begin{equation}\label{generating function}
H_j(x,x')=\frac{1}{2}(x-x')^2+v_j(x).
\end{equation} where $v_j(x+1)=v_j(x)$.

We start by summarizing the main concepts of the origin generalized
Frenkel-Kontorova model \eqref{variational problem} in the sense of
calculus of variations.

According to \cite{Morse'24}, we have
\begin{Definition}
A configuration $u: \mathbb{Z}^d \rightarrow \mathbb{R} $
 is
called a class-A minimizer for \eqref{variational problem} when for
every $\varphi: \mathbb{Z}^d \rightarrow \mathbb{R}$ with $\varphi_i
= 0$ when $|i| \ge N$, we have
\begin{equation}\label{ground state}
\sum_{i\in\mathbb{Z}^d, |i| \le N+1 }\sum_{j=1}^d
    H_{j}(u_{i},u_{i+e_j} )
\leq
\sum_{i\in\mathbb{Z}^d, |i| \le N+1 }\sum_{j=1}^d
    H_{j}(u_{i}+ \varphi_i ,u_{i+e_j} + \varphi_{i + e_j} ).
\end{equation}
\end{Definition}

The equation \eqref{ground state} can be interpreted heuristically
as saying $\mathscr{L}(u) \leq \mathscr{L}(u + \varphi)$ after
canceling the terms on both sides that are identical.

Class-A minimizers are also called \emph{ground states} in the
mathematical physics literature and  \emph{local minimizers} in the
calculus of variations literature.

\begin{Definition}
We say that a configuration is a critical point of the action
whenever it satisfies the Euler-Lagrange equations:
\begin{equation}\label{EulerLagrange}
\sum_{j=1}^d \partial_1 H_j( u_i, u_{i + e_j})  +
\partial_2 H_j( u_{i - e_j}, u_i)  = 0\qquad \text{for every}~i \in
\mathbb{Z}^d.
\end{equation}
\end{Definition}

For every given rotation vector $\omega\in\mathbb{R}^d$, we
introduce the corresponding Percival Lagrangian \eqref{Percivalian}
defined on the space of functions:
\begin{equation}
Y=\{~h~|~h~\text{monotone},~h(\theta+1)=h(\theta)+1,~h(\theta_-)=h(\theta)\}
\end{equation} where $h(\theta_-)$ (or $h(\theta_+)$) denote the left (or right)
limit of $h$ at point $\theta$.

We endow $Y$ with the graph topology defined in the following way.

Denote
\[{\text{graph}}(h)=\{(\theta,y)\in\mathbb{R}^2:h(\theta)\leq y\leq
h(\theta_+)\}.\]

If $h, \tilde{h} \in Y$ we define the distance between $h$ and
$\tilde{h}$ as the Hausdorff distance of their graphs:
\begin{equation} \label{graphdistance}
d(h,\tilde{h})=\max \{\sup_{ \xi \in {\text{graph}}(h)} \rho( \xi,
{\text{graph}}(\tilde{h})),
  \sup_{\eta \in {\text{graph}}(\tilde h)}\rho(\eta,{\text{graph}}(h))\}
\end{equation}
 where
 $\rho(\cdot,\cdot)$ is the Euclidean distance from a point to a
set, $\rho(x, S) = \inf_{y \in S} | y -x| $. Note that the graph
topology is weaker than the $L^\infty$ topology.

The formal variation of $\mathscr{L}_\omega$ yields the following
Euler-Lagrange equation for the hull function $h$:
\begin{equation}\label{Euler-Lagrange equation for hull functions}
 \sum_{j=1}^d [\partial_1 H_j ( h(\theta), h(\theta + \omega_j)) + \partial_2 H_j( h(\theta - \omega_j) , h(\theta))] = 0.
\end{equation}

We have the following theorem proved in \cite{SL} (see
\cite{Mather'82} in the case of twist maps):
\begin{Theorem}
For any rotation vector $\omega\in \mathbb{R}^d$, there exists a
minimizer $h_\omega\in Y$ of $\mathscr{P}_\omega$ satisfying
\eqref{Euler-Lagrange equation for hull functions}. Moreover, the
configurations defined by \eqref{parameterization} are the ground
states of \eqref{variational problem}.
\end{Theorem}

The description of the ground states are determined by certain
properties of $h_\omega$. In particular, we have
\begin{Definition}
The Frenkel-Kontorova model admits an invariant torus with rotation
vector $\omega$ if $h_\omega$ is a continuous function.
\end{Definition}

\section{A variational criterion of total destruction of invariant tori}
In this section, we will give a criterion of existence of invariant
tori independent of the rotation vector. First of all, it is easy to
prove the following lemma.
\begin{Lemma}\label{estimate lemma}
Let $h_1,~h_2\in Y$, then we have
\begin{equation}
\left|\int_0^1[(h_1(\theta+\omega_j)-h_1(\theta))^2-(h_2(\theta+\omega_j)-h_2(\theta))^2]d\theta\right|
\leq
2\int_0^1|h_1(\theta)-h_2(\theta)|d\theta.
\end{equation}
\end{Lemma}
\begin{proof}
We set
\[
\Delta(\theta)=h_1(\theta+\omega_j)-h_1(\theta)+h_2(\theta+\omega_j)-h_2(\theta),
\]
then $\Delta$ is periodic with range included in the interval
$[0,2]$. Moreover,
\begin{equation}
\begin{split}
&\left|\int_0^1[(h_1(\theta+\omega_j)-h_1(\theta))^2-(h_2(\theta+\omega_j)-h_2(\theta))^2]d\theta\right|\\
= & \left|\int_0^1\Delta(\theta) \cdot [(h_1(\theta+\omega_j)-h_2(\theta+\omega_j))-(h_1(\theta)-h_2(\theta))]d\theta\right|\\
= & \left|\int_0^1\Delta(\theta)\cdot(h_1(\theta+\omega_j)-h_2(\theta+\omega_j))-\int_0^1\Delta(\theta)\cdot(h_1(\theta)-h_2(\theta))d\theta\right|\\
= & \left|\int_0^1\Delta(\theta-\omega_j)\cdot(h_1(\theta)-h_2(\theta))-\int_0^1\Delta(\theta)\cdot (h_1(\theta)-h_2(\theta))d\theta\right|\\
= & \left|\int_0^1(\Delta(\theta-\omega_j)-\Delta(\theta))(h_1(\theta)-h_2(\theta))d\theta\right|\\
\leq & ~~2\int_0^1|h_1(\theta)-h_2(\theta)|d\theta.
\end{split}
\end{equation}
\end{proof}

From \eqref{generating function} and Lemma \ref{estimate lemma}, we
obtain a necessary condition of existence of invariant tori.
\begin{Lemma}\label{estimates for the potential}
If the generalized Frenkel-Kontorova models admits an invariant
tori, then there exists an continuous  $h_1\in Y$ and a $j\in
\{1,\ldots,d\}$ such that for every  discontinuous $h_2\in Y$, we
have
\begin{equation}
\int_0^1[v_j(h_1(\theta))-v_j(h_2(\theta))]d\theta \leq \int_0^1
|h_1(\theta)-h_2(\theta)|d\theta.
\end{equation}
\end{Lemma}
\begin{proof}
We assume by contradiction that there exists a non-continuous
function $h\in Y$ such that for every $j\in \{1,\ldots,d\}$
\begin{equation}
\int_0^1[v_j(h_1\theta))-v_j(h_2(\theta))]d\theta
> \int_0^1 |h_1(\theta)-h_2(\theta)|d\theta.
\end{equation}
Due to Lemma \ref{estimate lemma}, we obtain
\begin{equation}
\sum_{j=1}^d\int_0^1
[(h_2(\theta+\omega_j)-h(\theta))^2-(h_1(\theta+\omega_j)-h_1(\theta))^2]d\theta
< 2\sum_{j=1}^d\int_0^1 [v_j(h_1(\theta))-v_j(h_2(\theta))]d\theta
\end{equation}by summing over the index $j$.
From \eqref{Percivalian} and \eqref{generating function}, we have
\begin{equation}
\mathscr{L}_\omega(h) < \mathscr{L}_\omega(h_\omega).
\end{equation}
Since $h$ is non-continuous, it is contradicted by the existence of
invariant tori.
\end{proof}

\section{Proof of Theorem \ref{fundamental lemma}}
Based on the preparations above, we will prove Theorem
\ref{fundamental lemma} in this section. First of all, we give the
construction of $H_j^n$.

\subsection{Construction of the generalized Frenkel-Kontorova models}
Consider a sequence of nearly integrable systems generated by
\begin{equation}\label{H}
H_j^n(x,x')=\frac{1}{2}(x-x')^2+v_j^n(x)
\end{equation}where $v_j^n$ is a non-negative function satisfying

\begin{equation}\label{assumptions on the potential}
\begin{cases}
v_j^n(x+1)=v_j^n(x),\\
\max v_j^n(x)=v_j^n(x_0)=\frac{1}{n},\\
\text{supp}\,v_j^n\cap[0,1]=B_R(x_0),\\
R=\left(\frac{1}{n}\right)^{\frac{1}{r}},
\end{cases}
\end{equation}
where $n\in\mathbb{N}$, $x_0=\frac{1}{2}$ and $r$ is a positive
constant independent of $n$. From \eqref{assumptions on the
potential}, it is easy to see that
\begin{equation}\label{vv}
||v_j^n||_{C^{r-\epsilon}}\rightarrow 0\quad\text{as}\ n\rightarrow
\infty.
\end{equation}

\subsection{Total destruction of invariant tori}
From Lemma \ref{estimates for the potential}, it suffices to find a
discontinuous function $h_2\in Y$ such that for every continuous
function $h_1\in Y$, we have
\begin{equation}\label{non-existenc condtion}
\int_0^1 [v_j^n(h_1(\theta))-v_j^n(h_2(\theta))]d\theta > \int_0^1
|h_1(\theta)-h_2(\theta)|d\theta.
\end{equation}

Without loss of generality, one can assume $h_1(0)=0$. To achieve
this, we construct $h_2$ in the following way. It is sufficient to
define it on $[0,1)$ due to the fact that
$h_2(\theta+1)=h_2(\theta)+1$.

We set
\begin{equation}
h_2(\theta)=\left\{\begin{array}{ll}
\hspace{-0.4em}h_1(\theta),&0\leq \theta \leq A,\\
\hspace{-0.4em}x_0-R,
&A<\theta\leq B,\\
\hspace{-0.4em}h_1(\theta),&B< \theta<1,
\end{array}\right.
\end{equation}
where $R=n^{-\frac{1}{r}}$ and $A=h_1^{-1}(x_0-R)$ if
$h_1^{-1}(x_0-R)$ is a single point or $A=\min h_1^{-1}(x_0-R)$
otherwise and $B=h_1^{-1}(x_0+R)$ if $h_1^{-1}(x_0+R)$ is a single
point or $B=\max h_1^{-1}(x_0+R)$ otherwise.

Based on the construction of $v_j^n$, we have that for $\theta\in
[0,1)\backslash [A, B]$
\begin{equation}\label{ep}|h_1(\theta)-\theta|\leq \epsilon(n),\end{equation}where
$\epsilon(n)\rightarrow 0$  as $n\rightarrow \infty$. Indeed, by
(\ref{H}), for every rotation vector $\omega$,
\[H_j^n(h_1(\theta),h_1(\theta+\omega))=\frac{1}{2}(h_1(\theta)-h_1(\theta+\omega))^2+v_j^n(h_1(\theta)).\]
For $\theta\in [0,1)\backslash [A,B]$,
\[h_1(\theta)\not\in\text{supp}\,v_j^n,\] hence,
$v_j^n(h_1(\theta))=0$. It follows that (\ref{ep}) is verified.

Moreover, a simple calculation implies that
\begin{equation}
\int_0^1|h_1(\theta)-h_2(\theta)|d\theta \leq
C_2\left(\frac{1}{n}\right)^{\frac{2}{r}}.
\end{equation}
By \eqref{assumptions on the potential}, we have
\begin{equation}
\int_0^1v_j^n(h_2(\theta))d\theta=0,
\end{equation}
moreover,
\begin{equation}
\int_0^1v_j^n(h_1(\theta))-v_j^n(h_2(\theta))d\theta=\int_0^1v_j^n(h_1(\theta))d\theta
= O(\left(\frac{1}{n}\right)^{1+\frac{1}{r}}).
\end{equation}
To obtain \eqref{non-existenc condtion}, it is enough to require
\begin{equation}
\frac{2}{r}>1+\frac{1}{r}.
\end{equation}
Hence we get $r<1$. From (\ref{vv}) and Lemma \ref{estimates for the
potential}, we complete the proof of Theorem \ref{fundamental
lemma}.

\noindent\textbf{Acknowledgement} We would like to Prof. R. de la
Llave for his very useful comments which helped improve the
presentation and exposition. We also warmly thank Prof. C-Q.Cheng
for many helpful discussions. This work is under the support of the
National Basic Research Programme of China (973 Programme,
2007CB814800) and Basic Research Programme of Jiangsu Province,
China (BK2008013).

\bibliographystyle{alpha}
\bibliography{reference}

\begin{thebibliography}{KdlLR97}

\bibitem[ALD83]{Aubry'83}
S.~Aubry and P.~Y. Le~Daeron.
\newblock The discrete {F}renkel-{K}ontorova model and its extensions. {I}.
  {E}xact results for the ground-states.
\newblock {\em Phys. D}, 8(3):381--422, 1983.

\bibitem[Ban87a]{Bangert'87}
V.~Bangert.
\newblock The existence of gaps in minimal foliations.
\newblock {\em Aequationes Math.}, 34(2-3):153--166, 1987.

\bibitem[Ban87b]{Bangert'87a}
V.~Bangert.
\newblock A uniqueness theorem for {${\bf Z}^n$}-periodic variational problems.
\newblock {\em Comment. Math. Helv.}, 62(4):511--531, 1987.

\bibitem[Bla89]{Blank'89}
M.~L. Blank.
\newblock Metric properties of minimal solutions of discrete periodical
  variational problems.
\newblock {\em Nonlinearity}, 2(1):1--22, 1989.

\bibitem[CdlL98]{Rafael'98}
A.~Candel and R.~de~la Llave.
\newblock On the {A}ubry-{M}ather theory in statistical mechanics.
\newblock {\em Comm. Math. Phys.}, 192(3):649--669, 1998.

\bibitem[CdlL09]{CdlL'09}
Renato Calleja and Rafael de~la Llave.
\newblock Fast numerical computation of quasi-periodic equilibrium states in
  1{D} statistical mechanics, including twist maps.
\newblock {\em Nonlinearity}, 22(6):1311--1336, 2009.

\bibitem[Che11]{Cheng}
Chong~Qing Cheng.
\newblock Non-existence of {KAM} torus.
\newblock {\em Acta Math. Sin. (Engl. Ser.)}, 27(2):397--404, 2011.

\bibitem[Chi79]{Chirikov'79}
Boris~V. Chirikov.
\newblock A universal instability of many-dimensional oscillator systems.
\newblock {\em Phys. Rep.}, 52(5):264--379, 1979.

\bibitem[dlL01]{Rafael'01}
Rafael de~la Llave.
\newblock A tutorial on {KAM} theory.
\newblock In {\em Smooth ergodic theory and its applications ({S}eattle, {WA},
  1999)}, volume~69 of {\em Proc. Sympos. Pure Math.}, pages 175--292. Amer.
  Math. Soc., Providence, RI, 2001.

\bibitem[dlLV07]{Rafael'07}
Rafael de~la Llave and Enrico Valdinoci.
\newblock Ground states and critical points for generalized
  {F}renkel-{K}ontorova models in {$\Bbb Z\sp d$}.
\newblock {\em Nonlinearity}, 20(10):2409--2424, 2007.

\bibitem[For94]{Forni'94}
Giovanni Forni.
\newblock Analytic destruction of invariant circles.
\newblock {\em Ergodic Theory Dynam. Systems}, 14(2):267--298, 1994.

\bibitem[Her83]{H2}
Michael-R. Herman.
\newblock {\em Sur les courbes invariantes par les diff\'eomorphismes de
  l'anneau. {V}ol. 1}, volume 103 of {\em Ast\'erisque}.
\newblock Soci\'et\'e Math\'ematique de France, Paris, 1983.
\newblock With an appendix by Albert Fathi, With an English summary.

\bibitem[Her86]{H3}
Michael-R. Herman.
\newblock Sur les courbes invariantes par les diff\'eomorphismes de l'anneau.
  {V}ol.\ 2.
\newblock {\em Ast\'erisque}, (144):248, 1986.
\newblock With a correction to: {{\i}t On the curves invariant under
  diffeomorphisms of the annulus, Vol. 1} (French) [Ast{\'e}risque No. 103-104,
  Soc. Math. France, Paris, 1983; MR0728564 (85m:58062)].

\bibitem[Her0s]{H4}
Michael-R. Herman.
\newblock Non existence of lagrangian graphs.
\newblock 1990s.
\newblock Preprint.

\bibitem[KdlLR97]{Rafael'97}
Hans Koch, Rafael de~la Llave, and Charles Radin.
\newblock Aubry-{M}ather theory for functions on lattices.
\newblock {\em Discrete Contin. Dynam. Systems}, 3(1):135--151, 1997.

\bibitem[Mat82]{Mather'82}
John~N. Mather.
\newblock Existence of quasiperiodic orbits for twist homeomorphisms of the
  annulus.
\newblock {\em Topology}, 21(4):457--467, 1982.

\bibitem[Mat88]{Mather'88}
John~N. Mather.
\newblock Destruction of invariant circles.
\newblock {\em Ergodic Theory Dynam. Systems}, 8$^*$(Charles Conley Memorial
  Issue):199--214, 1988.

\bibitem[Mor24]{Morse'24}
Harold~Marston Morse.
\newblock A fundamental class of geodesics on any closed surface of genus
  greater than one.
\newblock {\em Trans. Amer. Math. Soc.}, 26(1):25--6 0, 1924.

\bibitem[Mos86]{Moser'86}
J{\"u}rgen Moser.
\newblock Minimal solutions of variational problems on a torus.
\newblock {\em Ann. Inst. H. Poincar\'e Anal. Non Lin\'eaire}, 3(3):229--272,
  1986.

\bibitem[Per79]{Percival'79}
I.~C. Percival.
\newblock A variational principle for invariant tori of fixed frequency.
\newblock {\em J. Phys. A}, 12(3):L57--L60, 1979.

\bibitem[Per80]{Percival'80}
I.~C. Percival.
\newblock Variational principles for invariant tori and cantori.
\newblock In {\em Nonlinear dynamics and the beam-beam interaction ({S}ympos.,
  {B}rookhaven {N}at. {L}ab., {N}ew {Y}ork, 1979)}, volume~57 of {\em AIP Conf.
  Proc.}, pages 302--310. Amer. Inst. Physics, New York, 1980.

\bibitem[SdlL11]{SL}
Xifeng Su and Rafael de~la Llave.
\newblock Percival lagrangian approach to aubry-mather theory.
\newblock {\em Expositiones Mathematicae}, 2011.
\newblock In Press and Preprint available at http://arxiv.org/abs/1104.2636.

\bibitem[Sli99]{Sinisa'99}
S.~Slijep{\v{c}}evi{\'c}.
\newblock Monotone gradient dynamics and {M}ather's shadowing.
\newblock {\em Nonlinearity}, 12(4):969--986, 1999.

\bibitem[SZ89]{SZ'89}
Dietmar Salamon and Eduard Zehnder.
\newblock K{AM} theory in configuration space.
\newblock {\em Comment. Math. Helv.}, 64(1):84--132, 1989.

\bibitem[Tak71]{Takens'71}
Floris Takens.
\newblock A {$C^{1}$} counterexample to {M}oser's twist theorem.
\newblock {\em Nederl. Akad. Wetensch. Proc. Ser. A {\bf 74}=Indag. Math.},
  33:378--386, 1971.

\bibitem[Wan11]{W}
Lin Wang.
\newblock Variational destruction of invariant circles.
\newblock {\em Discrete and Continuous Dynamical Systems}, 2011.
\newblock In Press.

\end{thebibliography}
\end{document}